\documentclass{amsproc}
\usepackage{mathabx}
\usepackage{mathtools}

\usepackage{}
\copyrightinfo{2009}{American Mathematical Society}

\newtheorem{theorem}{Theorem}[section]

\usepackage{tikz}
\usepackage{calc}
\usetikzlibrary{decorations.markings}
\usetikzlibrary{arrows}
\usetikzlibrary{arrows.meta}
\tikzstyle{vertex}=[circle, draw, inner sep=0pt, minimum size=1pt]

\theoremstyle{definition}

\newtheorem{rk}{Remark}[section]

\numberwithin{equation}{section}

\begin{document}

\title{Adjacency Matrix and Energy of the Line Graph of $\Gamma(\mathbb{Z}_n)$}

\author{Sheela Suthar}
\address{Department of Mathematics\\ Baba Farid Group of Institutions, Bathinda, Punjab - 151 001, India}
\curraddr{}
\email{sheelasuthar@gmail.com}
\thanks{}

\author{Om Prakash*}
\address{Department of Mathematics \\
IIT Patna, Bihta campus, Bihta-801 106, India}
\curraddr{}
\email{om@iitp.ac.in}
\thanks{}



\date{}

\maketitle
\begin{abstract}
Let $\Gamma(\mathbb{Z}_n)$ be the zero divisor graph of the commutative ring $\mathbb{Z}_n$ and $L(\Gamma(\mathbb{Z}_n))$ be the line graph of $\Gamma(\mathbb{Z}_n)$. In this paper, we discuss about the neighborhood of a vertex, the neighborhood number and the adjacency matrix of $L(\Gamma(\mathbb{Z}_n))$. We also study Wiener index and energy of $L(\Gamma(\mathbb{Z}_n))$, where $n=pq$, $p^2$ respectively for $p$ and $q$ are primes. Moreover, we give MATLAB coding of our calculations.

\end{abstract}

\section{INTRODUCTION}

Let $\mathbb{Z}_n$ be the commutative ring of residue classes $modulo$ a positive integer $n$ and $Z(\mathbb{Z}_n)$, the set of zero-divisors of $\mathbb{Z}_n$, i. e., $Z(\mathbb{Z}_n)$= $\{x\in \mathbb{Z}_n : xy= 0$, for some $y\in \mathbb{Z}_n\}$. The zero-divisor graph of $\mathbb{Z}_n$ is an undirected graph $\Gamma(\mathbb{Z}_n)$ with vertex set $Z(\mathbb{Z}_n)$ such that distinct vertices $x$ and $y$ are adjacent if and only if $xy= 0$. The line graph $L(\Gamma(\mathbb{Z}_n))$ of $\Gamma(\mathbb{Z}_n)$ is defined to be the graph whose vertex set constitutes of the edges of $\Gamma(\mathbb{Z}_n)$, where two vertices are adjacent if the corresponding edges have a common vertex in $\Gamma(\mathbb{Z}_n)$. The importance of line graphs stems from the fact that the line graph transforms the adjacency relation on edges to adjacency relation on vertices \cite{hph}.\\

The neighborhood (or open neighborhood) $N_{L(\Gamma(\mathbb{Z}_n))}(v)$ of a vertex $v$ of $L(\Gamma(\mathbb{Z}_n))$ is the set of vertices adjacent to $v$. The closed neighborhood \\ $N_{L(\Gamma(\mathbb{Z}_n))}[v]$ of a vertex $v$ is the set  $N_{L(\Gamma(\mathbb{Z}_n))}(v) \bigcup\{v\}$. A set of vertices $A$ is a neighborhood set if $L(\Gamma(\mathbb{Z}_n)) = \bigcup\limits_{v\in A}<N_{L(\Gamma(\mathbb{Z}_n))}[v]>$. The neighborhood number $n(L(\Gamma(\mathbb{Z}_n)))$ of a graph $L(\Gamma(\mathbb{Z}_n))$ is equal to the minimum number of vertices in a neighborhood set of $L(\Gamma(\mathbb{Z}_n))$. Moreover, $\Gamma(\mathbb{Z}_n)$ is the zero-divisor graph of a commutative ring $\mathbb{Z}_n$ with vertex set $Z(Z_n)$, therefore, $\Gamma(\mathbb{Z}_n)$ is always connected and so is $L(\Gamma(\mathbb{Z}_n))$ and hence $N_{L(\Gamma(\mathbb{Z}_n))}(V)$ is equal to the cardinality of $V$. Here, $\Delta(L(\Gamma(\mathbb{Z}_n)))$ and $\delta(L(\Gamma(\mathbb{Z}_n)))$ represent the maximum degree and minimum degree of $L(\Gamma(\mathbb{Z}_n))$ respectively. The adjacency matrix of $\Gamma(\mathbb{Z}_n)$ is defined by a matrix $A= [a_{ij}]$, where $a_{ij}= 1$, if $v_i$ and $v_j$ joined by an edge for any vertices $v_i$ and $v_j$ of $L(\Gamma(\mathbb{Z}_n))$ and $a_{ij}= 0$, otherwise. For basic definitions and results, we refer \cite{dp99,ks92,par11}. \\

In this paper, we calculate neighborhood and adjacency matrix for $L(\Gamma(\mathbb{Z}_n))$ in some cases. Also, we discuss the energy and Wiener index for the same. In graph theory, the energy of a graph is the sum of absolute values of the eigenvalues of the adjacency matrix. If we denote the length of shortest path between every pair of vertices $x, y \in V(L(\Gamma(\mathbb{Z}_n)))$ by $d(x, y)$, then the Wiener index \[ W(T(\Gamma(\mathbb{Z}_n))) = \frac{1}{2} \sum_{x,y \in V} d(x,y).\] An integral graph is a graph whose all eigenvalues are integers. In other words, a graph is an integral graph if all the roots of its characteristic equation are integers.

\section{Neighborhood and Adjacency Matrix for $pq$ and $p^2$}

Here, first we illustrate simple examples for neighborhood of vertices and adjacency matrix  for $n = pq$ and then present the generalize results.\\

\noindent {\bf Case 1:} When $p= 2$ and $q= 3$, we have $\mathbb{Z}_6$.\\
The ring $\mathbb{Z}_6$ has $3$ non-zero zero divisors. In this case, line graph has set of vertices $V = \{(2,3), (3,4)\}$ and line graph of $G = L(\Gamma(\mathbb{Z}_6))$ is \\

\begin{figure}[h]
	\centering
	\includegraphics[width=2in, height=0.5in]{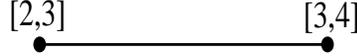}
	\caption{$L(\varGamma(\mathbb{Z}_6))$}
	\label{fig:1}
\end{figure}

The closed neighborhoods of the vertices are $N_G[(2,3)] = $ $\{(2,3), (3,4)\}$; $N_G[(3,4)] = $ $\{(3,4), (2,3)\}$. The neighborhood of $V$ is $N_G[(V)] = $ $\{(2,3), (3,4)\}$. The maximum degree is $\Delta(G) = 1$ and minimum degree is $\delta(G) = 1$. The adjacency matrix for the line graph of $G = L(\Gamma(\mathbb{Z}_6))$ is

\[
M_1 =
\left[ {\begin{array}{cc}
	0 & 1 \\
	1 & 0
	\end{array} } \right]_ {2\times 2}.
\] \\

{\bf Properties of adjacency matrix $M_1$:}
\begin{enumerate}
	\item{The determinant of the adjacency matrix $M_1$ corresponding to $G = L(\Gamma(\mathbb{Z}_6))$ is non-zero.}
	\item{The rank of the adjacency matrix $M_1$ corresponding to $G = L(\Gamma(\mathbb{Z}_6))$ is $2$.}
	\item{The adjacency matrix $M_1$ corresponding to $G = L(\Gamma(\mathbb{Z}_6))$ is symmetric and non-singular.}
	\item{ Trace of the adjacency matrix $M_1$ corresponding to $G = L(\Gamma(\mathbb{Z}_6))$ is zero.}
\end{enumerate}

\noindent {\bf Case 2:} When $p= 3$ and $q= 5$, we have $\mathbb{Z}_{15}$.

The ring $\mathbb{Z}_{15}$ has $6$ non-zero zero divisors. In this case, line graph has the set of vertices $V = \{(3,5), (9,5), (6,5), (5,12), (6,10), (3,10), (9,10), (10, 12)\}$ and line graph of $G = L(\Gamma(\mathbb{Z}_{15}))$ is given by
\begin{figure}[h]
	\centering
	\includegraphics[width=3in, height=2.1in]{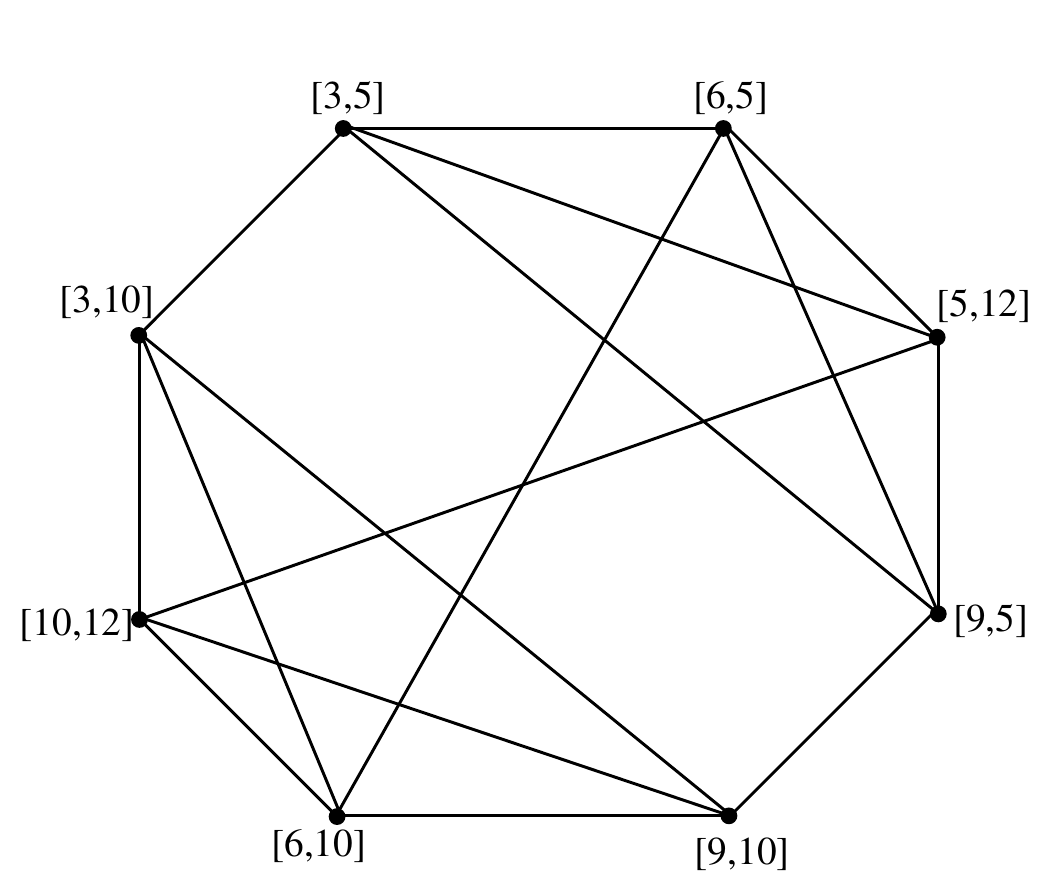}
	\caption{$L(\Gamma(\mathbb{Z}_{15}))$}
	\label{fig:2}
\end{figure}

The closed neighborhoods of the vertices are as follows:\\ $N_G[(3,5)]$ =  $\{(3,5), (6,5),
(9,5), (5,12), (3,10)\}$; $N_G[(9,5)]$ =  $\{(9,5), (6,5),\\ (3,5), (9,10), (5,12)\}$;
$N_G[(6,5)]$ = $\{(6,5), (3,5), (9,5), (6,10), (5,12)\}$;\\
$N_G[(5,12)]$ = $\{(5,12), (9,5), (6,5), (3,5), (10,12)\}$;
$N_G[(6,10)]$ = $\{(6,10),\\ (6,5), (9,10), (3,10), (10,12)\}$;
$N_G[(3,10)]$ = $\{(3,10), (10,12), (6,10), (9,10),\\ (3,5)\}$;
$N_G[(9,10)]$ = $\{(9,10), (10,12), (6,10), (3,10), (9,5)\}$;
$N_G[(10,12)]$ = \\$\{(10,12), (5,12), (6,10), (3,10), (9,10)\}$.
The neighborhood of $V$ is given by\\ $N_G[(V)] = $ $\{(3,5), (9,5), (6,5), (5,12), (6,10), (3,10), (9,10), (10, 12)\}$.\\ The maximum degree is $\Delta(G) = 4$ and minimum degree is $\delta(G) = 4$. The adjacency matrix for the line graph of $G = L(\Gamma(\mathbb{Z}_{15}))$ is \\
\[
M_1=
\begin{bmatrix}
m_1 & \hspace{1cm}  m_2 \\
m_2 & \hspace{1cm}  m_1
\end{bmatrix}_{8 \times 8}  \quad where, \hspace{.2cm}
m_1=
\begin{bmatrix}
0 & 1 & 1 & 1 \\
1 & 0 & 1 & 1 \\
1 & 1 & 0 & 1 \\
1 & 1 & 1 & 0
\end{bmatrix}_{4 \times 4}
\]

$m_2=
\begin{bmatrix}
1 & 0 & 0 & 0 \\
0 & 1 & 0 & 0 \\
0 & 0 & 1 & 0 \\
0 & 0 & 0 & 1
\end{bmatrix}_{4 \times 4}$, where $m_1$ is Symmetric matrix and $m_2$ is $4\times 4$ identity matrix.\\

{\bf Properties of adjacency matrix $M_1$ corresponding to $G = L(\Gamma(\mathbb{Z}_{15}))$:}\\

\begin{enumerate}
	\item{The determinant of the adjacency matrix $M_1$ corresponding to $G = L(\Gamma(\mathbb{Z}_{15}))$ is zero.}
	
	\item{The rank of the adjacency matrix $M_1$ corresponding to $G = L(\Gamma(\mathbb{Z}_{15}))$ is $5$.}
	
	\item{The adjacency matrix $M_1$ corresponding to $G = L(\Gamma(\mathbb{Z}_{15}))$ is symmetric and singular.}
	
	\item{ Trace of the adjacency matrix $M_1$ corresponding to $G = L(\Gamma(\mathbb{Z}_{15}))$ is zero.}
\end{enumerate}

\noindent {\bf Case 3:} When $p= 5$ and $q= 5$, we have $\mathbb{Z}_{25}$.\\
The ring $\mathbb{Z}_{25}$ has $4$ non-zero zero-divisors. In this case, line graph has $6$ vertices $V = \{(5,10),\\ (5,15), (5,20), (10,15), (10,20), (15,20)\}$ and line graph $G = L(\varGamma(\mathbb{Z}_{25}))$ is

\begin{figure}[h]
	\centering
	\includegraphics[width=2.5in, height=1.9in]{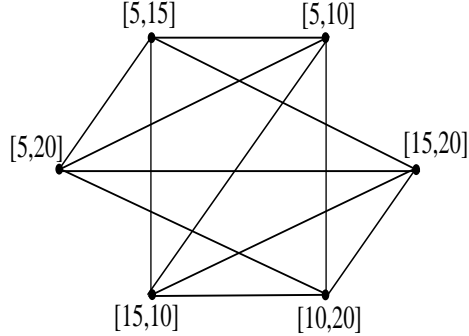}
	\caption{$L(\varGamma(\mathbb{Z}_{25}))$}
	\label{fig:4}
\end{figure}

The closed neighborhoods of the vertices are $N_G[(5,10)]$ =  $\{(5,10), (5,15),\\
(5,20), (10,15),(10,20)\}$; $N_G[(5,15)]$ =  $\{(5,10), (5,15), (5,20), (10,15),(15,20)\}$;\\ $N_G[(5,20)]$ =  $\{(5,10), (5,15), (5,20), (10,20), (15,20)\}$; $N_G[(10,15)]$ =  $\{(5,10),\\
(5,15), (10,20), (10,15), (15,20)\}$; $N_G[(10,20)]$ =  $\{(5,10), (5,20), (10,15), (10,20),\\
(15,20)\}$; $N_G[(15,20)]$ =  $\{(5,15), (5,20), (10,15),(10,20), (15,20)\}$. The neighborhood of the vertex set $V$ is given by $N_G[(V)] = $ $\{(5,10), (5,15), (5,20), (10,15),\\ (10,20), (15,20)\}$. The maximum degree is $\Delta(G) = 4$ and minimum degree is $\delta(G) = 4$. Therefore, the adjacency matrix for the line graph of $G = L(\Gamma(\mathbb{Z}_{25}))$ is \\
\[
M_1=
\begin{bmatrix}
m_1 & \hspace{1cm}  m_2 \\
m_2 & \hspace{1cm}  m_1
\end{bmatrix}_{6 \times 6}  \quad where, \hspace{.2cm}
m_1=
\begin{bmatrix}
0 & 1 & 1  \\
1 & 0 & 1  \\
1 & 1 & 0
\end{bmatrix}_{3 \times 3}
\]

$m_2=
\begin{bmatrix}
1 & 1 & 0  \\
1 & 0 & 1 \\
0 & 1 & 1
\end{bmatrix}_{3 \times 3}$, $m_1$ and $m_2$ are Symmetric  $3\times 3$ matrix.\\\\

{\bf Properties of adjacency matrix $M_1$ corresponding to $G = L(\Gamma(\mathbb{Z}_{25}))$:}\\

 \begin{enumerate}
	\item{ The determinant of the adjacency matrix $M_1$ corresponding to $G = L(\Gamma(\mathbb{Z}_{25}))$ is zero.}
	
	\item{The rank of the adjacency matrix $M_1$ corresponding to $G = L(\Gamma(\mathbb{Z}_{25}))$ is $3$.}
	
	\item{The adjacency matrix $M_1$ corresponding to $G = L(\Gamma(\mathbb{Z}_{25}))$ is symmetric and singular.}
	
	\item{ Trace of the adjacency matrix $M_1$ corresponding to $G = L(\Gamma(\mathbb{Z}_{25}))$ is zero.}
\end{enumerate}

\section{Energy and Wiener index for $pq$ and $p^2$}

\begin{theorem} If $p = 2$ and $q$ is an odd prime number, then energy of $L(\Gamma(\mathbb{Z}_{2q}))$ is $2q - 4$.
\end{theorem}

\begin{proof} Let $p$ = $2$ and $q$ be an odd prime number. Then $L(\Gamma(\mathbb{Z}_{2q}))$ is a complete graph with $(q-1)$ vertices and every vertex is adjacent to every other vertex of $\mathbb{Z}_{2q}$. Therefore, \\
	\[
	M(L(\Gamma(\mathbb{Z}_{2q})))=
	\begin{bmatrix}
	0 & 1 & 1 \  \cdots & 1 & 1\\
	1 & 0 &  1 \  \cdots & 1 & 1\\
	\vdots & \vdots & &  & \vdots \\
	1 & 1 &  1 \ \cdots & 1 & 0
	\end{bmatrix}_{q-1 \times q-1}.
	\] \\
	
	Hence, the characteristic polynomial of the above matrix is\\
	
	$f(\lambda)$ = $\mid  M(L(\Gamma(\mathbb{Z}_{2q}))) - \lambda I_{q-1} \mid$ = $ \begin{vmatrix}
	0-\lambda & 1 & \cdots & 1 \\
	1 & 0-\lambda & \cdots & 1 \\
	\vdots & \vdots &  & \vdots \\
	1 & 1 & \cdots & 0-\lambda
	\end{vmatrix}$
	
	\vspace{.2in}
	\hspace{1.9in} = $(\lambda + 1)^{q-2}(\lambda - q + 2).$\\
	
	So, characteristic equation $f(\lambda)$ = 0, has the roots $\lambda  = -1$, $q-2$. Therefore, non-zero eigenvalues are $(q-2)$ times $-1$ and one time $(q-2)$. Thus, $\sum_i^{q-1}\mid \lambda_i\mid = 2q-4.$
\end{proof}

\begin{theorem} If $p$ = $3$ and $q > 3$ is any prime number, then energy of $L(\Gamma(\mathbb{Z}_{3q}))$ is $4q - p - 5$.
\end{theorem}

\begin{proof}Let $p$ = $3$ and $q$ be any prime number greater than $3$. Then  $L(\Gamma(\mathbb{Z}_{3q}))$ is $(q-1)$-regular graph with $\mid \frac{1}{4} \sum degree(v_i) \mid = 2(q-1)$ vertices. Therefore,\\
	\[
	M(L(\Gamma(\mathbb{Z}_{3q})))=
	\begin{bmatrix}
	M_1 &    M_2 \\
	M_2 &    M_1
	\end{bmatrix},
	\]\\
	where $M_1$ is a symmetric matrix of $(q-1)\times (q-1)$ with all diagonal entries zero and non-diagonal entries are one and $M_2$ is the identity matrix of $(q-1)\times (q-1)$. Hence, the characteristic polynomial of the above matrix is \\
	
	$f(\lambda)$ = $\mid  M(L(\Gamma(\mathbb{Z}_{3q}))) - \lambda I_{q-1} \mid$ = $\lambda^{q-2} (2+ \lambda)^{q-2}(p - q + \lambda)(1- q + \lambda)$.\\
	
	The characteristic equation of the matrix is $f(\lambda)$ = 0, and the characteristic roots are $\lambda  = 0, -2, (q-p)$ and $(q-1)$. Therefore, non-zero eigenvalues are $(q-2)$ times $-2$, one time $(q-p)$ and one time $(q-1)$. Thus,\\
	
	\hspace{1in} $E(L(\Gamma(\mathbb{Z}_{3q})))$ = $2(q-2) + q - p + q - 1$\\
	
	\hspace{1.8in} = $2q - 4 + 2q - p -1$\\
	
	\hspace{1.8in} = $4q - p - 5$.
\end{proof}
\begin{theorem} If $p$ and $q > 2$ are distinct primes, then energy of $L(\Gamma(\mathbb{Z}_{pq}))$ is $4pq - 8p - 8q + 16$.
\end{theorem}
\begin{proof} Let $p$ and $q$ be distinct prime numbers greater than $2$. Then the line graph of $\Gamma(\mathbb{Z}_{pq})$ is a $(p + q - 4)$-regular graph with $(p - 1)(q - 1)$ vertices. Here, $M$ is a matrix of order $(p -1)(q - 1)$ and the characteristic polynomial of $M$ is \\
	
	\hspace{.6in} $f(\lambda)$ = $\mid  M(L(\Gamma(\mathbb{Z}_{pq}))) - \lambda I_{(p-1)(q-1)} \mid$\\
	
	\hspace{.8in} = $(2+ \lambda)^{(p-2)(q-2)}(p - 3 - \lambda)^{q-2} (q - 3 - \lambda)^{p-2}( (q + p - 4 - \lambda)$.\\
	
	The characteristic equation of the matrix is $f(\lambda)$ = 0, and roots are $\lambda  = -2, (p - 3), (q - 3)$ and $(q + p - 4)$. Therefore, non-zero eigenvalues are $(p - 2)(q - 2)$ times $-2$, $(q - 2)$ times $(p - 3)$, $(p - 2)$ times $(q - 3)$ and one time $(q + p - 4)$. Thus,\\
	
	\hspace{.2in} $E(L(\Gamma(\mathbb{Z}_{pq})))$ = $2(p-1)(q-2) + (q-2)(p-3) + (p-2)(q-3) + p + q - 4$\\
	
	\hspace{.6in} = $2pq - 4p - 4q + 8 + pq - 3q - 2p + 6 + pq - 3p - 2q + 6 + p + q - 4$\\
	
	\hspace{.6in} = $4pq - 8p - 8q - 16$.
\end{proof}

\begin{theorem} Let $n$ = $p^2$, where $p\geq5$ be any prime number. Then $E(L(\Gamma(\mathbb{Z}_{p^2})))$ = $2p^2 - 10p + 8$.
\end{theorem}

\begin{proof} Let $n = p^2,$ $p\geq5$ be any prime number. Then $L(\Gamma(\mathbb{Z}_{p^2}))$ is $2(p-3)$-regular graph with $\dfrac{(p-1)(p-2)}{2}$ vertices. Therefore,\\
\[
M(L(\Gamma(\mathbb{Z}_{p^2})))=
\begin{bmatrix}
0 & 1 & 1 & \cdots & 0\\
1 & 0 & 1 & \cdots & 1\\
\cdots & \cdots & \cdots & \cdots & \cdots\\
\cdots & \cdots & \cdots & \cdots & 0\\
\end{bmatrix},
\]\\
where $M$ is symmetric matrix of $\dfrac{(p-1)(p-2)}{2}\times \dfrac{(p-1)(p-2)}{2}$ with all diagonal entries are zero and non-diagonal entries are one if the vertices are adjacent or zero if they are non-adjacent. Therefore, the characteristic polynomial of above matrix is \\

\hspace{1.4in}  $f(\lambda)$ = $\lvert  M(L(\Gamma(\mathbb{Z}_{p^2}))) - \lambda I_{\frac{(p-1)(p-2)}{2}} \lvert$\\

\hspace{1.7in} = $(\lambda-2p+6) (\lambda-p+5)^{(p-2)}(\lambda+2)^\frac{(p-1)(p-4)}{2}$.\\

The characteristic equation of the matrix is $f(\lambda)$ = 0,  and eigenvalues are one time $2(p-3)$, $(p-2)$ times $(p-5)$ and $\frac{(p-1)(p-4)}{2}$ times $(-2)$ respectively. Hence, \\

\hspace{.8in} $E(L(\Gamma(\mathbb{Z}_{p^2})))$ = $2(p-3) + (p-2)(p-5) + 2\dfrac{(p-1)(p-4)}{2} $ \\

\hspace{1.6in} = $2p -6 + p^2 - 5p - 2p + 10 + p^2 - p - 4p + 4$\\

\hspace{1.6in} = $2p^2 - 10p + 8$.
\end{proof}

\begin{theorem} If $ p = 2$ and $q$ is an odd prime number, then $W(L(\Gamma(\mathbb{Z} _{2q})))$ is $\dfrac{(q-1)(q-2)}{2}$.
\end{theorem}

\begin{proof} Let $p$ = $2$ and $q$ be an odd prime number. Then $L(\Gamma(\mathbb{Z}_{2q}))$ is a complete graph with $(q-1)$ vertices and every vertex is adjacent to every other vertex of line graph of zero divisor graph of $\mathbb{Z}_{2q}$. Hence, for every vertex of line graph of $\Gamma(\mathbb{Z}_{2q})$ has distance one from other vertices of graph, i.e. $d(x,y)=1$, $x$, $y\in L(\Gamma(\mathbb{Z}_{2q}))$. Therefore,\\
\[W(L(\Gamma(\mathbb{Z}_{2q}))) = \sum d(x,y) = \sum\limits_{x \neq y}1 \]\\
\vspace{0.1in}
\hspace{2.1in} $ = \begin{pmatrix}
q-1 \\
2
\end{pmatrix}$ = $\dfrac{(q-1)(q-2)}{2}$.
\end{proof}

\begin{theorem} If $ p = 3$ and $q>3$ is any prime number, then $W(L(\Gamma(\mathbb{Z}_{3q})))$ is $(q-1)(3q-p-2)$.
\end{theorem}

\begin{proof} Let $p$ = $3$ and $q>3$ be any prime number greater than $3$. Then $L(\Gamma(\mathbb{Z}_{3q}))$ is $(q-1)$-regular graph with $2(q-1)$ vertices and every vertex is adjacent with $(q-1)$ vertices of $L(\Gamma(\mathbb{Z}_{3q}))$. Therefore, distance of a vertex from other vertex is one if they are adjacent and two if they are not adjacent. So, we can divide the set of vertices of $L(\Gamma(\mathbb{Z}_{3q}))$ into two sets, say $T$ and $S$, where $T$ and $S$ contain those vertices $u$ and $v$ such that $d(u, v) = 1$ and $d(u, v) = 2$ respectively. For every $u\in T$,\\
\[W(u) = \sum\limits_{u\in T}d(u,v) +  \sum\limits_{u\in S, u\neq v}d(u,v)\]
\vspace{0.1in}
\[= \sum\limits_{u\in T}1 +  \sum\limits_{u\in S, u\neq v}d(u,v)\]\\
\vspace{0.1in}
\hspace{1.6in} = $2(q-1)(q-1) + (q-1)(q-p).$\\
For every $u\in S$,
\vspace{0.1in}
\[W(u) = \sum\limits_{u\in S}d(u,v) +  \sum\limits_{u\in T, u\neq v}d(u,v)\]
\vspace{0.1in}
\[= \sum\limits_{u\in S}1 +  \sum\limits_{u\in T, u\neq v}d(u,v)\]\\
\vspace{0.1in}
\hspace{1.6in} = $2(q-1)(q-1) + (q-1)(q-p).$\\
Therefore,\\
\[W(L(\Gamma(\mathbb{Z}_{3q}))) = \frac{1}{2} \left[ {\sum\limits_{u\in L} W(u) + \sum\limits_{u\in S} W(u)} \right] \]
\vspace{0.1in}
\hspace{.5in} = $\dfrac{1}{2}[2(q-1)(q-1) + (q- 1)(q-p) + 2(q-1)(q-1) + (q-1)(q-p)]$\\
\vspace{0.1in}
\hspace{.5in} = $(q-1)[2(q-1) + (q-p)]$\\
\vspace{0.1in}
\hspace{.5in} = $(q-1)(3q-p-2)$.
\end{proof}

\begin{theorem} If $p$, $q$ are prime numbers such that $2 < p < q$, then $W(L(\Gamma(\mathbb{Z}_{pq})))$ is $\dfrac{1}{2}(p-1)(q-1)(2pq - 3p - 3q + 4)$.
\end{theorem}

\begin{proof} Let $p$ and $q$ be prime numbers and $2 < p < q$. Then  $L(\Gamma(\mathbb{Z}_{pq}))$ is $(p + q - 4)$-regular graph with $(p-1)(q-1)$ vertices. So, we can divide all the vertices of $L(\Gamma(\mathbb{Z}_{pq}))$ into two sets, say $T$ and $S$, where $T$ and $S$ contain those vertices $u$ and $v$ such that $d(u, v) = 1$ and $d(u, v) = 2$ respectively. For every $u\in T$,
\[W(u) = \sum\limits_{u\in T}d(u,v) +  \sum\limits_{u\in S, u\neq v}d(u,v)\]
\vspace{0.1in}
\[= \sum\limits_{u\in T}1 +  \sum\limits_{u\in S, u\neq v}2\]\\
\vspace{0.1in}
\hspace{.3in} = $\dfrac{(p-1)(q-1)}{2}(p + q - 4) + 2\dfrac{(p-1)(q-1)}{2}[(p-1)(q-1) - (p+q-3)].$\\

\hspace{.11in} = $\dfrac{(p-1)(q-1)}{2}[(p + q - 4) + 2\{(p-1)(q-1) - (p+q-3)\}].$\\

For every $u\in S$,
\vspace{0.1in}
\[W(u) = \sum\limits_{u\in S}d(u,v) +  \sum\limits_{u\in T, u\neq v}d(u,v)\]
\vspace{0.1in}
\[= \sum\limits_{u\in S}1 +  \sum\limits_{u\in T, u\neq v}2\]\\
\vspace{0.1in}
\hspace{.3in} = $\dfrac{(p-1)(q-1)}{2}(p + q - 4) + 2\dfrac{1(p-1)(q-1)}{2}[(p-1)(q-1) - (p+q-3)]$\\

\hspace{.11in} = $\dfrac{(p-1)(q-1)}{2}[(p + q - 4) + 2\{(p-1)(q-1) - (p+q-3)\}].$\\

Hence, \[W(L(\Gamma(\mathbb{Z}_{pq}))) = \frac{1}{2} \left[ {\sum\limits_{u\in L} W(u) + \sum\limits_{u\in S} W(u)} \right] \]
\vspace{0.1in}
\hspace{0in} = $\dfrac{1}{2}[\dfrac{(p-1)(q-1)}{2}[(p + q - 4) + 2\{(p-1)(q-1) - (p+q-3)\}]+$\\
\vspace{0in}
\hspace{.4in} $\dfrac{(p-1)(q-1)}{2}[(p + q - 4) + 2\{(p-1)(q-1) - (p + q - 3)\}]]$\\

\vspace{0.1in}
\hspace{.5in}= $\dfrac{(p-1)(q-1)}{2}[p + q - 4 + 2\{(p-1)(q-1) - (p+q-3)\}]$\\

\vspace{0.1in}
\hspace{.5in}= $\dfrac{(p-1)(q-1)}{2}[p + q - 4 + 2(pq-p-q+1-p-q+3)]$\\

\vspace{0.1in}
\hspace{.5in}= $\dfrac{(p-1)(q-1)}{2}(2pq -3p -3q +4)$.
\end{proof}

\begin{theorem}  If $ n = p^2$ and $p\geq5$ is any prime number, then $W(L(\Gamma(\mathbb{Z}_{p^2}) ))$ is $\dfrac{1}{4}(p-1)(p-2)^2 (p-3)$.
\end{theorem}

\begin{proof} Let $n$ = $p^2$, $p\geq5$ be any prime number. Then $L(\Gamma(\mathbb{Z}_{p^2}))$ is $2(p-3)$-regular graph with $\dfrac{(p-1)(p-2)}{2}$ vertices and every vertex is adjacent with $2(p-3)$ vertices of the line graph of the zero divisor graph of $\mathbb{Z}_{p^2}$. Hence, distance of a vertex from other vertex is one if they are adjacent and  it is two if they are not adjacent. Therefore, we can divide the set of vertices  of $L(\Gamma(\mathbb{Z}_{p^2}))$ into two sets, say $T$ and $S$, where $T$ and $S$ contain those vertices $u$ and $v$ such that $d(u, v) = 1$ and $d(u, v) = 2$ respectively. We use same procedure as above Theorem [3.7] to find the distance between any two vertices and finally we get the Wiener index of $L(\Gamma(\mathbb{Z}_{p^2}))$ is $\dfrac{1}{4}(p-1)(p-2)^2 (p-3)$.
\end{proof}

\begin{theorem} Let $\mathbb{Z}_n$ be the finite commutative ring and $A$ the adjacency matrix corresponding to the line graph of its zero divisor graph. Then trace of $A$ is always zero.
\end{theorem}

\begin{proof} Let $[a_{i, j} ]_{n\times n}$ be adjacency matrix corresponding to the line graph of zero divisor graph of the finite commutative ring $\mathbb{Z}_n$. As we know, trace \[ A= \sum_{i = j} a_{i, j}.\] In adjacency matrix, diagonal entry will be non-zero due to self loop., i.e., $v_{i}.v_{j} = 0,$ so $a_{i,j}$, $i =1, 2, 3, \cdots $ will be non-zero. Since, above discussed examples shows that no vertex of the line graph of zero divisor graph over ring $\mathbb{Z}_n$ has self loop. Therefore, all the diagonal entries are zero. Hence, $tr(A) = 0$.
\end{proof}

\begin{theorem} In $L(\Gamma(\mathbb{Z}_n))$, if $n= pq$ such that $ p < q$, then neighbourhood number $n(L(\Gamma(\mathbb{Z}_n))) = p-1$.
\end{theorem}

\begin{proof} Let $n = pq$ be the product of primes and $p<q$. Then $\Gamma(\mathbb{Z}_n)$ is a bipartite graph $K_{p-1, q-1}$ such that its line graph $L(\Gamma(\mathbb{Z}_n))$ is a regular graph of $(p-1)(q-1)$ vertices and degree of each vertex is $p+q+4$. The vertices of $L(\Gamma(\mathbb{Z}_n))$ are
	\begin{center}
		$(p, q), (2p, q), (3p, q),...,((q-1)p, q)$\\
		$(p, 2q), (2p, 2q), (3p, 2q),...,((q-1)p, 2q)$\\
		\hspace{-5cm}.\\
		\hspace{-5cm}.\\
		\hspace{-5cm}.\\
		$(p, (p-1)q), (2p, (p-1)q), (3p, (p-1)q),...,((q-1)p, (p-1)q),$
	\end{center}
	in which for every fixed $m$ and $1 \leq n \leq q-1$, $(np, mq)$ forms a complete graph and also the line graph $L(\Gamma(\mathbb{Z}_n))$ is connected. Hence, the neighbourhood number $n(L(\Gamma(\mathbb{Z}_n))) = p-1$.
	
\end{proof}

 \begin{theorem} If $n$ = $pq$ or $p^2$, then $L(\Gamma(\mathbb{Z}_n))$ is an integral graph.
\end{theorem}

\begin{proof} {\bf Case(1):} If $q>2$ is any prime number and $p = 2$, then the line graph $L(\Gamma(\mathbb{Z}_{2q}))$ is a complete graph with $(q-1)$ vertices. In this case $L(\Gamma(\mathbb{Z}_{2q}))$ has adjacency matrix $A$ = $J - I$, where matrix $J$ has all entries $1$ and $I$ is the identity matrix. Since, $j$ has spectrum $(q-2)^1$, $0^{q-2}$ and $I$ has spectrum $-1^{q-1}$ and $IJ = JI$, it follows that $L(\Gamma(\mathbb{Z}_{2q}))$ has spectrum $(q-2)^1$, $(-1)^{q-2}$. This proves that all eigenvalues are integers.\\
	
	\noindent {\bf Case(2):} If $q > 3$ is a prime number and $p = 3$, then $L(\Gamma(\mathbb{Z}_{3q}))$ is a $(q-1)$-regular graph with $2(q-1)$ vertices. Also, each row of the adjacency matrix $A$ of $L(\Gamma(\mathbb{Z}_{3q}))$ has $(q-1)$ ones. Therefore, $A$ has spectrum $0^{q-2}$, $(-2)^{q-2}$, $(q-p)^1$ and $(q-1)^1$. Clearly all the eigenvalues are integers and multiplicities of these eigenvalues are non-negative integers.\\
	
	Therefore, $L(\Gamma(\mathbb{Z}_{n}))$ is an integral graph for $n = 2q$ and $3q$, where $q$ is any prime number. Similarly, we continue above procedure and finally prove that $L(\Gamma(\mathbb{Z}_{n}))$ is an integral graph for $n = pq$, where $p$ and $q$ are distinct prime numbers.\\
	
	\noindent {\bf Case(3):} If $n = p^2$, $p$ is a prime number, then $L(\Gamma(\mathbb{Z}_{p^2}))$ is a $2(p-3)$-regular graph with $\dfrac{(p-1)(p-2)}{2}$ vertices. Also, each row of the adjacency matrix $A$ of $L(\Gamma(\mathbb{Z}_{p^2}))$ has $2(p-3)$ ones, so that spectrum are $2(p-3)$, $(p-5)^{p-2}$ and $(-2)^\frac{(p-1)(p-4)}{2}$. Clearly, all eigenvalues have integeral values. Hence, $L(\Gamma(\mathbb{Z}_{p^2}))$ is an integral graph.
\end{proof}

\begin{rk}
	Note that our graphs are undirected, therefore, the matrices are symmetric and all eigenvalues are real and non-fractional. So, all eigenvalues of graph $L(\Gamma(\mathbb{Z}_n))$, for $n = pq$ and $p^{2}$ are integers. Moreover, our graphs do not have loops so that the diagonal entries of matrices are zero. Thus, trace of the matrix is zero and sum of eigenvalues is also zero.
\end{rk}

\section{Matlab coding of Line Graph}
Here, we present an algorithm for constructing the line graph of $\Gamma(\mathbb{Z}_n)$ and for calculating energy and Wiener index of this graph. This algorithm contains several sub-algorithms. It is sufficient to input $n$. In the first stage of algorithm, we get matrix of the line graph and plot $L(\Gamma(\mathbb{Z}_n)) = A$ by function $LG(p)$. At the second stage, we calculate Energy index by using $E = sum(abs(eg))$ and Wiener index by function $ Wiener(A)$.\\

$function A = LG(p)$

$n = p;$

$M = [];$

$for \ i = 1:n-1$

$for \ j = 1:n-1$

$if \ mod(i*j,n)==0$

$M = [M,i];$

$break;$

$end$

$end$

$end$

$M$

$n = length(M);$

$Nz = zeros(n);$

$for \ i = 1:n-1$

$for \ j = i+1:n$

$if \ mod(M(i)*M(j),p)==0$

$Nz(i,j) = 1; Nz(j,i) = 1;$

$end$

$end$

$end$

$[row col] = find(triu(Nz))$

$Points = [row col]$

$[nr nc] = size(Points)$

$A = zeros(nr);$

$for \ i = 1:length(A)$

$axes(i+1,:) = [cos(pi*i/nr),\ sin(pi*i/nr)];$

$end$

$hold on$

$for \ i = 1:nr$

$for \ j = 1:nr$

$if \ i ~= j$

$plot(axes(i,1),\ axes(i,2),\ '*')$

$if intersect(Points(i,:),\ Points(j,:)) ~= 0$

$A(i,j) = 1;$

$plot(axes([i,j],1),\ axes([i,j],2));$

$end$

$end$

$end$

$end$

$eg = eig(A);$

$E = sum(abs(eg))$

$Dist = Wiener(A)$

$sum(sum(Dist))/2$

$function s = Wiener(B)$

$B(B==0) = inf;$

$A = triu(B,1)+\ tril(B,-1);$

$m = length(A);$

$B = zeros(m);$

$j = 1;$

$while \ j <= m$

$for \ i = 1:m$

$for \ k = 1:m$

$B(i,k) = min(A(i,k),\ A(i,j)+\ A(j,k));$

$end$

$end$

$A = B;$

$j = j+1;$

$end$

$s = A;$
\section{ Conclusion } In this paper, we have studied the adjacency matrix, neighborhood and related properties such as determinant, rank, trace and eigenvalues for  $L(\Gamma(\mathbb{Z}_n))$, where $n$ is the products of primes. Also, we have discussed the energy and Wiener index for $pq$ and $p^2$ of the line graph of zero divisor graph of ring $\mathbb{Z}_n$. Furthermore, we have given some formulae for Wiener index and energy of line graph and found that above discussed graphs for $n$ = $pq$ and $p^2$ are integral graphs.

\section*{Acknowledgment}
The authors are thankful to Council of Scientific and Industrial Research (CSIR), Govt. of India for financial support under Ref. No. 22/06/2014(i)EU-V, Sr. No. 1061440753 dated 29/12/2014 and Indian Institute of Technology Patna for providing the research facilities.

\end{document}